\documentclass[11pt]{article}

\usepackage{graphicx}
\usepackage{latexsym}
\usepackage[usenames,dvipsnames]{color}
\usepackage{amssymb,amsmath}
\usepackage{lscape}
\usepackage{verbatim}
\usepackage{epic}
\usepackage{bm}

\normalbaselines
\catcode `\@=11
\@addtoreset{equation}{section}
\def\theequation{\arabic{section}.\arabic{equation}}
\def\section{\@startsection {section}{1}{\z@}{-3.5ex plus -1ex minus -.2ex}{2.3ex plus .2ex}{\large\bf}}
\def\subsection{ \@startsection{subsection}{2}{\z@}{-3.25ex plus -1ex minus -.2ex}{1.5ex plus .2ex}{\normalsize \bf}}
\def\subsubsection{\@startsection{subsubsection}{3}{\z@}{-3.25ex plus -1ex minus -.2ex}{1.5ex plus .2ex}{\normalsize\sl}}
\catcode`\@=12

\newtheorem{theorem}{Theorem}[section]
\newtheorem{defn}[theorem]{Definition}
\newtheorem{prop}[theorem]{Proposition}
\newtheorem{lemma}[theorem]{Lemma}
\newtheorem{cor}[theorem]{Corollary}

\newtheorem{rems}[theorem]{Remarks}
\newtheorem{rem}[theorem]{Remark}
\newtheorem{example}[theorem]{Example}

\newenvironment{proof}{\noindent \bf Proof : \rm}{$ \hspace{\stretch{1}} \Box $}

 \newcommand{\bea}{\begin{eqnarray}}
\newcommand{\ena}{\end{eqnarray}}

\newcommand{\beano}{\begin{eqnarray*}}
\newcommand{\enano}{\end{eqnarray*}}

\newcommand{\bei}{\begin{itemize}}
\newcommand{\eni}{\end{itemize}}

\newcommand{\be}{\begin{equation}}
\newcommand{\en}{\end{equation}}

\newcommand{\bedefi}{\begin{defn} \rm }
\newcommand{\findefi}{\end{defn}}

\newcommand{\belem}{\begin{lemma}}
\newcommand{\enlem}{\end{lemma}}

\newcommand{\beprop}{\begin{prop}}
\newcommand{\enprop}{\end{prop}}

\newcommand{\betheo}{\begin{theorem}}
\newcommand{\entheo}{\end{theorem}}

\newcommand{\becor}{\begin{cor}}
\newcommand{\encor}{\end{cor}}

\newcommand{\berem}{\begin{rem} \rm}
\newcommand{\enrem}{\end{rem}}

\newcommand{\berems}{\begin{rems} \rm}
\newcommand{\enrems}{\end{rems}}

\newcommand{\beex}{\begin{example}$\!\!\!$ \rm }
\newcommand{\enex}{ \end{example}}

\newcommand{\norm}[2]{
\left\| #2 \right\|_{#1}
}
\newcommand{\Hil}[0]{\mathcal{H} }

\renewcommand{\le}{\leqslant}
\renewcommand{\leq}{\leqslant}
\renewcommand{\geq}{\geqslant}

 \newcommand{\noi}{\noindent}

 \newcommand{\ov}{\overline}

\newcommand{\RN}{\mathbb{R}}
\newcommand{\ZN}{\mathbb{Z}}
\newcommand{\CN}{{\mathbb C}}

\newcommand{\ran}{{\sf Ran}\,}
\newcommand{\Ker}{{\sf Ker}\,}
\newcommand{\hs}{Hilbert space}
\newcommand{\pip}{{\sc pip}-space}

 \newcommand{\ta}{^\times}

\def\C{{\mathcal C}}
\def\D{{\mathcal D}}

\def\J{{\mathcal J}}
\def\I{{\mathcal I}}
\def\H{{\mathcal H}}
\def\L{{\mathcal L}}
\def\M{{\mathcal M}}
\def\V{{\mathcal V}}

\def\P{{\mathfrak P}}

\def\T{\widehat{T}}
\newcommand{\VJ}{V_{_{\scriptstyle\rm J}}}
\newcommand{\VI}{V_{_{\scriptstyle\rm I}}}
\newcommand{\NN}[0]{\mathbb{N}}
\newcommand{\ud}{\,\mathrm{d}}
\newcommand{\up}{\raisebox{0.7mm}{$\upharpoonright$}}

\newcommand{\cl}[2]{[ {#1}]_{#2}}
\newcommand{\co}{^{\#}}
\newcommand{\ip}[2]{ \langle {#1} |{#2}  \rangle}
 \newcommand{\du}[2]{\langle {#1},{#2}\rangle}
\newcommand{\ipp}[2]{ {\bm\langle}\!\!\!{\bm\langle} {#1} |{#2}{\bm\rangle}\!\!\!{\bm\rangle}_{\scriptscriptstyle\mu}}

\newcommand{\mc}{\mathcal}

\begin{document}

\begin{flushleft}
{\Large \sc Reproducing pairs of measurable functions and partial inner product spaces} \vspace*{7mm}

{\large\sf   J-P. Antoine $\!^{\rm a}$ and C. Trapani $\!^{\rm b}$
}
\\[3mm]
$^{\rm a}$  {\small Institut de Recherche en Math\'ematique et  Physique, Universit\'e catholique de Louvain \\
\hspace*{3mm}B-1348   Louvain-la-Neuve, Belgium\\
\hspace*{3mm}{\it E-mail address}: jean-pierre.antoine@uclouvain.be}
\\[1mm]
$^{\rm b}$   {\small Dipartimento di Matematica e Informatica,
Universit\`a di Palermo, \\
\hspace*{3mm} I-90123 Palermo, Italy\\
\hspace*{3mm} {\it E-mail address}: camillo.trapani@unipa.it}
\end{flushleft}

 \begin{abstract}
We continue the analysis  of reproducing pairs of {weakly} measurable functions, which generalize  continuous frames. 
More precisely, we examine the case where the defining measurable functions take their values in a partial inner product space (\pip).
Several examples, both discrete and continuous, are presented. 
\medskip

\textbf{AMS classification numbers:} 41A99, 46Bxx, 46C50, 46Exx
\medskip


\end{abstract}

\section{Introduction}
\label{sec-intro}

Frames and their relatives are most  often considered in the discrete case, for instance in signal processing \cite{christ}.
However, continuous frames have also been studied and offer interesting mathematical problems. They have been introduced originally by Ali, Gazeau and one of us
\cite{squ-int,contframes} and also, independently, by Kaiser \cite{kaiser}. Since then, several papers dealt with various aspects of the
concept, see for instance
\cite{gab-han} or \cite{rahimi}.
 However, there may occur situations where it is impossible to satisfy both frame bounds.

Therefore, several generalizations of frames have been introduced. Semi-frames
\cite{ant-bal-semiframes1,ant-bal-semiframes2}, for example, are obtained when functions    only satisfy one of the two frame bounds. 
It turns out that a large portion of frame theory can be extended to this larger framework,
in particular the notion of duality.

More recently, a new generalization of frames was introduced by Balazs and Speckbacher \cite{speck-bal}, namely, reproducing
pairs. Here, given a measure space $(X,  \mu)$, one considers a couple of weakly measurable functions $(\psi, \phi)$,   instead of a single mapping,
and one studies  the correlation between the two (a precise definition is given below).
This definition also includes the original definition of a continuous frame \cite{squ-int,contframes} given the choice $\psi= \phi$.   
The increase of  freedom in choosing the mappings $\psi$ and $\phi$, however, leads to the problem of characterizing
the range of the analysis operators, which in general need no more be contained in $L^2(X, \ud \mu)$, as in the frame case. Therefore,
  we extend the theory to the case where the weakly measurable functions take their values in a partial inner product space (\pip). We discuss first the case of a rigged Hilbert space, then we consider a genuine \pip.  {We conclude with two natural families of examples, namely,  Hilbert scales and several \pip s generated by the family  $\{L^p(X, \ud \mu), 1\leq p \leq \infty\}$}.

 \section{Preliminaries}
 \label{sec-prel}
 
Before proceeding, we list our definitions and  conventions. The framework is
 a (separable) Hilbert space $\H$, with the inner product $\ip{\cdot}{\cdot}$ linear in the first factor.
Given an operator $A$ on $\H$, we denote its domain by $D(A)$, its range by $\ran(A)$ and its kernel by ${\Ker}(A)$.
$GL(\H)$ denotes the set of all invertible bounded operators on $\H$ with bounded inverse.
Throughout the paper, we will consider weakly measurable functions $\psi: X \to \H$, where $(X,\mu)$ is a  locally compact  space with a Radon measure $\mu$,  {that is, $\ip{\psi_x}{f}$ is $
\mu-$measurable for every $f\in\H$}.
Then the weakly measurable function $\psi$ is a \emph{continuous frame} if there exist constants  ${\sf m} > 0$  and ${\sf M}<\infty$  (the  frame bounds) such that
\be\label{eq:frame}
{\sf m}  \norm{}{f}^2 \leq    \int_{X}  |\ip{f}{\psi_{x}}| ^2 \, \ud \mu(x)  \le {\sf M}  \norm{}{f}^2 ,  \forall \, f \in \H.
\end{equation}
Given the continuous frame $\psi$, the  \emph{analysis} operator ${\sf C}_{\psi}: \Hil \to L^{2}(X, \ud\mu)$
\footnote{As usual, we identify a function $\xi$ with its residue class in  $L^{2}(X, \ud\mu)$.}
is defined as
\be\label{eq:csmap}
({\sf C}_{\psi}f)(x) =\ip{f} {\psi_{x}}, \; f \in \H,
\end{equation}
 and the corresponding \emph{synthesis operator} ${\sf C}_{\psi}^\ast: L^{2}(X, \ud\mu) \to \H$ as
 (the integral being understood in the weak sense,  as usual)
\be\label{eq:synthmap}
{\sf C}_{\psi}^\ast \xi =  \int_X  \xi(x) \,\psi_{x} \; \ud\mu(x), \mbox{ for} \;\;\xi\in L^{2}(X, \ud\mu).
\end{equation}
We set   $S:={\sf C}_{\psi}^* {\sf C}_{\psi}$, {which is self-adjoint.}

More generally, the couple of weakly measurable functions $(\psi, \phi)$   is called a \emph{reproducing pair} if \cite{ast-reprodpairs}
\\[1mm]
(a) The sesquilinear form
\be\label{eq:form}
\Omega_{\psi, \phi}(f,g) = \int_X \ip{f}{\psi_x} \ip{\phi_x}{g} \ud\mu(x)
\en
 is well-defined and bounded on $\H \times \H$, {that is, $| \Omega_{\psi, \phi}(f,g) | \leq c \norm{}{f}\norm{}{g}$,} {  for some $c>0$.
 \\[1mm]
(b) The corresponding bounded (resolution) operator $S_{\psi, \phi}$ belongs to $GL(\H)$.
\medskip

Under these hypotheses, one has
\be\label{eq-Sf}
S_{\psi, \phi}f = \int_X \ip{f}{\psi_x} \phi_x  \ud\mu(x) , \; \forall\,f\in\H,
\en
the integral on the r.h.s. being defined in weak sense.
If $\psi = \phi$, we recover the notion  of continuous frame, so that we have indeed a genuine generalization of the latter.
 {Notice that $S_{\psi, \phi}$ is in general neither positive, nor self-adjoint, since $S_{\psi, \phi}^\ast= S_{\phi, \psi}$ .
However, if  $\psi,\phi$  is  reproducing pair, then $\psi , S_{\psi, \phi}^{-1}\phi$ is a dual pair, that is, the corresponding resolution operator is the identity.
Therefore, there is no restriction of generality to assume that $S_{\phi, \psi}=I$ \cite{speck-bal}. 
 {The worst that can happen is to replace some norms by equivalent ones}.

 }

In \cite{ast-reprodpairs}, it has been shown that each weakly measurable function $\phi$ generates an intrinsic pre-Hilbert space $V_\phi(X, \mu)$ and, moreover, a reproducing pair $(\psi, \phi)$ generates two Hilbert spaces, 
$V_\psi(X, \mu)$ and $V_\phi(X, \mu)$, conjugate dual of each other with respect to the $L^2(X, \mu)$ inner product. Let us briefly sketch that construction, that we will generalize further on.

Given    a weakly measurable function $\phi$, let us denote by  $\V_\phi(X, \mu) $ the space of all  measurable functions  $\xi  : X \to \CN$ such that the integral $\int_X  \xi(x)  \ip{\phi_x}{g} \ud\mu(x)$
exists for every $g\in \H$ and defines a bounded conjugate linear functional on $\H$, i.e., $\exists\; c>0$ such that
\be\label{eq-Vphi}
\left|  \int_X  \xi(x)  \ip{\phi_x}{g} \ud\mu(x) \right| \leq c \norm{}{g}, \; \forall\, g \in \H.
\en
{Clearly, if  $(\psi,\phi)$ is a reproducing pair, all functions $\xi(x)  = \ip{f}{\psi_x} = (C_\psi f)(x)$ 
belong to  $\V_\phi(X, \mu) $.}

By the Riesz lemma, we can define a linear map $T_\phi : \V_\phi(X, \mu) \to \H$ by the following weak relation
\be\label{eq:Tphi2}
\ip{T_\phi   \xi}{g}    =\int_X  \xi(x)  \ip{\phi_x}{g} \ud\mu(x) , \; { \forall} \, \xi \in\V_\phi(X, \mu), g\in\H.
\en
Next, we   define the  vector space
$$
V_\phi(X, \mu)= \V_\phi(X, \mu)/{\Ker}\,T_\phi
$$
and equip it with the norm 
\be \label{eq-normphi}
\norm{\phi}{\cl{\xi}{\phi}} := \sup_{\norm{}{g}\leq 1 } \left|  \int_X  \xi(x)  \ip{\phi_x}{g} \ud\mu(x)  \right|
=  \sup_{\norm{}{g}\leq 1 }\left| \ip{T_\phi   \xi}{g} \right|,
\en
where we have put $\cl{\xi}{\phi}= \xi + {\Ker}\,T_\phi$ for $\xi\in \V_\phi(X, \mu)$.
Clearly, $V_\phi(X, \mu)$ is a normed space.
However, the norm $\norm{\phi}{\cdot}$ is in fact Hilbertian, that is, it derives from an inner product, as can be seen as follows. First, it turns out that the map $\T_\phi:V_\phi(X,\mu)\rightarrow \H$, $\T_\phi[\xi]_\phi:= T_\phi \xi$
is a well-defined isometry  of $V_\phi(X, \mu)$ into $\H$. Next, one may define on $V_\phi(X,\mu)$ an inner product by setting
$$
\ip{\cl{\xi}{\phi}}{\cl{\eta}{\phi}}_{(\phi)}: 
=\ip{\T_\phi\cl{\xi}{\phi}}{\T_\phi\cl{\eta}{\phi}},\; \cl{\xi}{\phi}, \cl{\eta}{\phi}{ \in V_\phi(X,\mu)},
$$
and one shows that the norm defined by $ \ip{\cdot}{\cdot}_{(\phi)}$ coincides with the norm $\| \cdot\|_\phi$ defined in \eqref{eq-normphi}. One has indeed
$$
\norm{(\phi)}{\cl{\xi}{\phi}}= \norm{}{\T_\phi\cl{\xi}{\phi}}= \norm{}{T_\phi\xi}=  \sup_{\norm{}{g}\leq 1 }\left| \ip{T_\phi   \xi}{g} \right|
= \norm{\phi}{\cl{\xi}{\phi}}.
$$ 
Thus $V_\phi(X, \mu) $ is a pre-Hilbert  space. 

With these notations, the main result of \cite{ast-reprodpairs} reads as
\betheo\label{theo-dual}
If $(\psi,\phi)$ is a reproducing pair,  the spaces $V_\phi(X, \mu)$ and $V_\psi(X, \mu)$ are both Hilbert spaces, {conjugate dual} of each other with respect to the sesquilinear form 
\be \label{eq_sesq}
\ipp{\xi}{\eta}:= \int_X \xi(x)  \ov{\eta(x) } \ud\mu(x),
\en  
\entheo
which coincides with the inner product of $L^2(X, \mu)$ whenever the latter makes sense. This is true, in particular, for
$\phi = \psi$, since then $\psi$ is a continuous frame and $V_\psi(X, \mu)$ is a closed subspace of $L^2(X, \mu)$.

In this paper, we will consider reproducing pairs in the context of \pip s. The motivation is the following.
Let $(\psi, \phi)$ be a reproducing pair. By definition,
\be\label{eq-defS}
\ip{S_{\psi,\phi}f}{g}=\int_X \ip{f}{\psi_x}\ip{\phi_x}{g}\ud\mu(x) =  \int_X C_\psi f (x) \; \ov{C_\phi g(x)}  \ud\mu(x)
\en
is well defined for all $f,g\in\H$. The r.h.s. coincides with the sesquilinear form \eqref{eq_sesq}, that is,
 the $L^2$ inner product, but generalized, since in general $C_\psi f ,C_\phi g$
need not belong to   $L^2(X,\ud\mu)$.
If, following \cite{speck-bal}, we make the innocuous assumption that $\psi$ is uniformly bounded,  i.e.,
$\sup_{x\in X}\norm{\H}{\psi_x } \leq c$ for some $c>0$
(often $\norm{\H}{\psi_x}$ = const., e.g. for wavelets or coherent states),
then $(C_\psi f)(x) = \ip{f}{\psi_x} \in L^\infty(X,\ud \mu)$.

These two facts suggest to take   $\ran C_\psi$ within some \pip\ {of measurable functions, possibly} related to the $L^p$ spaces. We shall present several possibilities in that direction in Section \ref{sec-Lp}.

\section{Reproducing pairs and RHS}
\label{sec-RHS}

We begin with the simplest example of a \pip, namely, a rigged Hilbert space (RHS). Let indeed $\D[t] \subset \H \subset \D^\times [t^\times]$ be a RHS with $\D [t]$ reflexive (so that $t$ and $t^\times$ coincide with the respective  Mackey topologies).
Given  a measure space $(X, \mu)$, we denote by $\du{\cdot}{\cdot}$ the sesquilinear form expressing the duality between $\D$ and $\D^\times$.
 As usual, we suppose that this  sesquilinear form extends the inner product of $\D$ (and $\H$). This allows to build the triplet above.
Let $x\in X \mapsto \psi_x, \,x\in X \mapsto \phi_x$ be weakly measurable functions from $X$ into $\D^\times$.

Instead of \eqref{eq:form}, we consider he sesquilinear form
\be\label{eq:formD}
\Omega^{\D}_{\psi, \phi}(f,g) = \int_X \du{f}{\psi_x} \du{\phi_x}{g} \ud\mu(x), \; f,g \in \D,
\en
and we assume that it is jointly continuous on $\D \times \D$, that is $\Omega^{\D}\in {\sf B}(\D,\D)$ in the notation of \cite[Sec.10.2]{ait-book}. Writing
\be\label{eq-def_altS}
\du{S_{\psi,\phi}f}{g}:=\int_X \du{f}{\psi_x}\du{\phi_x}{g}\ud\mu(x) , \; \forall\, f,g\in\D,
\en
we see that  the operator $S_{\psi,\phi}$ belongs to  $\L(\D, \D^\times)$, the space of all continuous linear maps from $D$ into  $\D^\times$.}

\subsection{A Hilbertian approach}
\label{subsec-hilb}

We first  assume that the sesquilinear form $\Omega^{\D}$
 is well-defined and bounded on $\D \times \D$ in the topology of $\H$.
 Then $\Omega^{\D}_{\psi, \phi}$ extends to a bounded sesquilinear form on $\H\times \H$, denoted by the same symbol.

The definition of the space $\V_\phi(X,\mu)$ must be modified as follows. Instead of \eqref{eq-Vphi}, we  suppose that the integral below exists and defines a  conjugate linear functional   on $\D$, bounded in the topology of $\H$, i.e.,
\be\label{eq-sesqform}
\left|  \int_X  \xi(x)  \du{\phi_x}{g} \ud\mu(x) \right| \leq c \norm{}{g}, \; \forall\, g \in \D.
\en
Then the functional  extends to  a bounded conjugate linear functional on $\H$, since $\D$ is dense in $\H$.
Hence, for every $\xi \in \V_\phi(X, \mu) $, there exists a unique vector $h_{\phi,\xi}\in \H$ such that
$$
\int_X  \xi(x)  \du{\phi_x}{g} \ud\mu(x) = \ip{h_{\phi,\xi}}{g}, \quad \forall g \in \D.
$$
It is worth remarking that this interplay between the two topologies on $\D$ is similar to the approach of Werner \cite{werner}, who treats $L^2$ functions as distributions, thus identifies the $L^2$ space as the dual of $\D = \C_0^\infty$ with respect to the norm topology. And, of course, this is fully in the spirit of \pip s.

 {Then, we can define a linear map $T_\phi :  \V_\phi(X, \mu) \to \H$ by
\be\label{eq:TphiD}
 T_\phi   \xi = h_{\phi,\xi}\in \H, \;\forall\,\xi\in \V_\phi(X, \mu),
\en
in the following weak sense
$$
\ip{T_\phi   \xi}{g} = \int_X  \xi(x)  \du{\phi_x}{g} \ud\mu(x) , \; g\in\D, \xi\in \V_\phi(X, \mu).
$$
In other words we are \emph{imposing} that $\int_X \xi(x)\phi_x \ud\mu(x)$ converge  weakly to an element of $\H$.}

 The rest proceeds as before. We consider the space $ V_\phi(X, \mu)= \V_\phi(X, \mu)/{\Ker}\,T_\phi$, with the norm
 $\norm{\phi}{\cl{\xi}{\phi}}=\norm{}{T_\phi \xi}$,  where, for $\xi\in V_\phi(X, \mu)$, we have put $\cl{\xi}{\phi}= \xi + {\Ker}\,T_\phi$. 
Then $ V_\phi(X, \mu)$  is a pre-Hilbert space for that norm. 

Note that $\phi$ was called in \cite{ast-reprodpairs} \emph{$\mu$-independent} whenever ${\Ker}\,T_\phi = \{0\}$. In that case, of course, $V_\phi = \V_\phi $.

Assume, in addition, that the corresponding  bounded operator $S_{\psi,\phi}$ is an element of $GL(\H)$.
Then $(\psi,\phi)$ is a reproducing pair and  Theorem  3.14 of  \cite{ast-reprodpairs} remains true, that is,
 \betheo\label{theo-dual1}
If $(\psi,\phi)$ is a reproducing pair,  the spaces $V_\phi(X, \mu)$ and $V_\psi(X, \mu)$ are both Hilbert spaces, {conjugate dual} of each other with respect to the sesquilinear form 
\be\label{eq-dual}
\ip{\cl{\xi} {\phi}} {\cl{\eta} {\psi}}= \int_X \xi(x)  \ov{\eta(x) } \ud\mu(x), \; \forall\, {\xi}  \in \V_\phi(X, \mu),\,  {\eta} \in \V_\psi(X, \mu).
\en
\entheo

 \beex To give a trivial example, consider the Schwartz rigged Hilbert space ${\mathcal S}(\RN) \subset L^2(\RN, \ud x)
 \subset {\mathcal S}^\times(\RN)$, $(X,\mu) = (\RN, \ud x)$,
 $\psi_x (t)= \phi_x(t)= \frac{1}{\sqrt{2\pi}} e^{ixt}$ . Then  {$C_\phi f= \widehat f$}, the Fourier transform,
 so that $\ip{f}{\phi(\cdot)}\in L^2(\RN, \ud x)$.
 In this case
 $$
 \Omega_{\psi,\phi}(f,g)= \int_\RN \du{f}{\psi_x}\du{\phi_x}{g}\ud x= \ip{\widehat f}{\widehat g}=\ip{f}{g}, \;\forall f,g \in {\mathcal S}(\RN),
 $$
 and $V_\phi(\RN,\ud x)=L^2(\RN, \ud x)$.
 \enex

\subsection{The general case}
\label{sec-alt}

In the general case, we only assume that the form $\Omega$ is jointly continuous on $\D \times \D$, with no other regularity requirement.   
{In that case,   the vector space  $\V_\phi(X, \mu)$ must be defined differently.
Let the topology of $\D$ be given by a {directed family ${\P}$ of   seminorms}.  Given    a weakly measurable function $\phi$, we denote again by  $\V_\phi(X, \mu) $ the space of all  measurable functions  
  $\xi  : X \to \CN$ such that the integral $\int_X  \xi(x)  \du{\phi_x}{g} \ud\mu(x)$
exists for every $g\in \D$ and defines a continuous conjugate linear functional on $\D$, namely, there exists $c>0$ and a seminorm ${\sf p} \in \P$ such that
$$
\left|\int_X  \xi(x)  \du{\phi_x}{g} \ud\mu(x)\right| \leq c  \, {\sf p}(g).
$$
This in turn determines a linear map $T_\phi : \V_\phi(X, \mu) \to \D^\times$ by the following   relation
\be\label{eq:Tphi3}
\du{T_\phi   \xi}{g}    =\int_X  \xi(x)  \du{\phi_x}{g} \ud\mu(x) , \; { \forall} \, \xi \in\V_\phi(X, \mu), g\in\D.
\en
Next, we   define as before the   vector space
$$
V_\phi(X, \mu)= \V_\phi(X, \mu)/{\Ker}\,T_\phi,
$$
and we  put again $\cl{\xi}{\phi}= \xi + {\Ker}\,T_\phi$ for $\xi\in \V_\phi(X, \mu)$.

Now we need to introduce a topology on $V_\phi(X, \mu)$. We proceed as follows.
{Let $\M$ be a bounded subset of $\D[t]$. Then we define
 \be
  {\widehat {\sf p}_\M}( \cl{\xi}\phi) := \sup_{g\in \M}\left| \du{T_\phi   \xi}{g} \right|.
   \en

{ 
That is, we are defining the topology of $V_\phi(X,\mu)$ by means of the strong dual topology $t^\times$ of $\D^\times$ which we recall   is defined by the seminorms
$$ 
\|F\|_\M = \sup_{g \in \M} |\ip{F}{g}, \quad F\in \D^\times,
$$
where $\M$ runs over the family of bounded subsets of $\D[t]$. As said above, the reflexivity of $\D$ entails that $ t^\times$ is equal to the Mackey topology $\tau(\D^\times,\D)$.
More precisely,
\belem
The map $\T_\phi:V_\phi(X,\mu)\rightarrow \D^\times$, $\T_\phi[\xi]_\phi:= T_\phi \xi$
is a well-defined linear map  of $V_\phi(X, \mu)$ into $\D^\times$ and, for every bounded subset $\M$ of $\D[t]$, one has
$${\widehat {\sf p}_\M}( \cl{\xi}\phi) =\|T_\phi \xi\|_\M, \quad \forall \xi \in \V_\phi(X, \mu)$$
\enlem
The latter equality obviously implies the continuity of $T_\phi$.

Next  we
investigate the dual {$V_\phi(X,\mu)^\ast$} of the space $V_\phi(X, \mu) $, that is, the set of continuous linear  functionals on $V_\phi(X, \mu) $.
First, we have to choose a topology for $V_\phi(X,\mu)^\ast$. As usual we take the strong dual topology. This is defined by the family of seminorms
$$ {\sf q}_{\mc R} (F):= \sup_{\cl{\xi}\phi \in {\mc R}}|F(\cl{\xi}\phi)|,$$ where ${\mc R}$ runs over the bounded subsets of $V_\phi(X, \mu) $.

\medskip

\betheo\label{theo23} Assume that $\D[t]$ is a reflexive space and
let $\phi$ be a  weakly measurable function.
If $F$ is a continuous linear functional on $V_\phi(X, \mu) $, then there exists a unique {$g\in\D$}  such that
\begin{equation}\label{repres-functional2}
F(\cl{\xi}{\phi}) =  \int_X  \xi(x)  {\du{\phi_x}{g}} \ud\mu(x) , \; \forall\, \xi \in V_\phi(X, \mu)
\end{equation}
Moreover, every $g\in \H$ defines a continuous functional  $F$ on $V_\phi(X,\mu)$ with {$\norm{\phi^\ast}{F}  \leq \norm{}{g}$},  by \eqref{repres-functional2}.
\entheo
\begin{proof}
Let $F\in V_\phi(X, \mu)^\ast$. Then, there exists a bounded subset $\M$ of $\D[t]$ such that
$$
| F(\cl{\xi}{\phi}) | \leq {\widehat {\sf p}_\M}( \cl{\xi}\phi)={\norm{}{T_\phi \xi}}_\M, \; \forall \, \xi \in \V_\phi(X, \mu).
$$
Let ${\sf M}_\phi:= \{ T_\phi \xi : \xi \in \V_\phi(X, \mu) \} = \ran  \T_\phi $. Then  ${\sf M}_\phi$ is a vector subspace of $\D^\times$.

Let $\widetilde F$ be the functional defined on ${\sf M}_\phi$ by
$$
{\widetilde F}(T_\phi \xi) := F(\cl{\xi}{\phi}), \; \xi \in \V_\phi(X, \mu).
$$
We  notice that $\widetilde F$ is well-defined. Indeed, if $T_\phi \xi = T_\phi \xi'$, then $\xi- \xi'\in {\Ker\,}T_\phi$. Hence, $\cl{\xi}{\phi}=\cl{\xi'}{\phi}$
and $F(\cl{\xi}{\phi})=F(\cl{\xi'}{\phi})$

Hence, $\widetilde F$ is a continuous linear functional   on ${\sf M}_\phi$ which can be extended (by the Hahn-Banach theorem) to a continuous linear functional on $\D^\times$. Thus, in virtue of the reflexivity of $\D$,  there exists a vector {$g\in\D$}  such that
$$
{\widetilde F}(T_\phi \xi)  = \du{\T_\phi \cl{\xi}{\phi}}  {g } = \du{T_\phi  \xi}{g } = \int_X  \xi(x)  {\du{\phi_x}{g}} \ud\mu(x) .
$$
In conclusion,
$$
F(\cl{\xi}{\phi}) =  \int_X  \xi(x)  {\du{\phi_x}{g}} \ud\mu(x),\; \forall\, \xi \in \V_\phi(X, \mu) .
$$
Moreover, every $g\in \D$ obviously defines a continuous linear functional $F$ on $V_\phi(X, \mu)$   by \eqref{repres-functional2}. In addition, if ${\mc R}$ is a bounded subset of $V_\phi(X, \mu) $, we have
\begin{align*}
{\sf q}_{\mc R} (F)  &= \sup_{\cl{\xi}\phi \in {\mc R}}|F(\cl{\xi}\phi)|=\sup_{\cl{\xi}\phi \in {\mc R}}\left| \int_X  \xi(x)  {\du{\phi_x}{g}} \ud\mu(x)\right|
\\
&= \sup_{\cl{\xi}\phi \in {\mc R}}|{\du{T_\phi   \xi}{g}}|\leq
\sup_{\cl{\xi}\phi \in {\mc R}}  {\widehat {\sf p}_\M}( \cl{\xi}\phi),
\end{align*}
for any bounded subset $\M$ of $\D$ containing $g$.
\end{proof}

\medskip

In the present context, the analysis operator $C_\phi$ is defined in the usual way, given in \eqref{eq:csmap}.
Then, particularizing the discussion of Theorem \ref{theo23} to the functional ${\du{C_\phi g}{\cdot}}$, one can interpret the analysis operator $C_\phi$ as a continuous operator from $\D$ to $V_\phi(X,\mu)^\ast$.
As in the case of frames or semi-frames, one may characterize the synthesis operator in terms of the analysis operator.

 \begin{prop}\label{prop210}
 {Let} $\phi$ be weakly measurable, then $\widehat T_\phi\subseteq C_\phi^\ast$. If, in addition, $V_\phi(X,\mu)$ is reflexive,
then $\T_\phi^\ast=C_\phi$. Moreover, $\phi$ is  $\mu$-total {(i.e. $\Ker C_\phi = \{0\}$)} if and only if
$\ran  \widehat T_\phi$ is dense in $\D^\times$.
\end{prop}

\begin{proof}
 {
As $C_\phi:\D\rightarrow V_\phi(X,\mu)^\ast$ is a continuous operator, it
 has a continuous adjoint
$C_\phi^\ast : V_\phi(X,\mu)^{\ast\ast} \to \H$ \cite[Sec.IV.7.4]{schaefer}. Let $C_\phi^\sharp :=  C_\phi^\ast \up V_\phi(X,\mu)$. Then $C_\phi^\sharp = \widehat T_\phi$ since,
for every $f\in\D$, $[\xi]_\phi\in V_\phi(X,\mu)$,
\be\label{eq-adjoint}
{\du{C_\phi f}{[\xi]_\phi}}
=\int_X {\du{f}{\phi_x}} \overline{\xi(x)}\ud\mu(x)={\du{f}{\widehat T_\phi [\xi]_\phi}}.
\en
If $ V_\phi(X,\mu)$ is reflexive,  we have, of course,  $C_\phi^\sharp =C_\phi^\ast = \widehat T_\phi$.}

If $\phi$ is not $\mu$-total, then there exists $f\in \D$, $f\neq0$ such that
$(C_\phi f)(x)=0$ for a.e. $x\in X$. Hence, $f\in (\ran \widehat T_\phi)^\bot:= \{f\in \D: \ip{F}{f}=0, \, \forall F\in \ran \widehat T_\phi\}$ by \eqref{eq-adjoint}.
 Conversely, if $\phi$ is $\mu$-total, as $(\ran \widehat T_\phi)^\bot=\Ker  C_\phi=\{0\}$, by the reflexivity of $\D$ and $\D^\times$, it follows that $\ran \widehat T_\phi$ is dense in $\D^\times$.
\end{proof}

\medskip

 In a way similar to what we have done above, we can define the space $V_\psi(X, \mu)$, its topology, the  {residue} classes $\cl{\eta}{\psi}$, the operator $T_\psi$, etc, replacing $\phi$ by $\psi$.
Then, $V_\psi(X, \mu)$ is a  locally convex space.

\betheo \label{representation-of-F}
Under the condition \eqref{eq:formD},  every bounded linear functional $F$ on $V_\phi(X, \mu)$,  {i.e.,  $F\in V_\phi(X, \mu)^\ast$,} can be represented as
\be\label{eq-dual2}
F(\cl{\xi}{\phi}) = \int_X \xi(x)  \ov{\eta(x) } \ud\mu(x), \; \forall\,\cl{\xi}{\phi}\in V_\phi(X, \mu),
\en
with $\eta\in \V_\psi(X, \mu)$. The  {residue} class $\cl{\eta}{\psi}\in V_\psi(X, \mu)$ is uniquely determined.
\entheo
\begin{proof}
By Theorem \ref{theo23}, we have the representation
$$
F(\xi) =  \int_X  \xi(x)  {\du{\phi_x}{g}} \ud\mu(x)  .
$$
It is easily seen that $\eta(x) = {\du{g}{\phi_x}} \in \V_\psi(X, \mu)$.

It remains to prove uniqueness. Suppose that
$$
F(\xi) = \int_X \xi(x)  \ov{\eta'(x) } \ud\mu(x) .
$$
Then
$$
\int_X \xi(x)  (\ov{\eta'(x)}- \ov{\eta(x)}) \ud\mu(x) =0.
$$
Now the function $\xi(x) $ is arbitrary. Hence, taking in particular for $\xi(x) $ the functions  $\du{f}{\psi_x} , f\in\D$,
we get $\cl{\eta}{\psi}=\cl{\eta'}{\psi}$.
\end{proof}

\medskip

The lesson of the previous statements is that the map
\be\label{map-j}
  j : F\in V_\phi(X, \mu)^\ast \mapsto \cl{\eta}{\psi} \in V_\psi(X, \mu)
\en
is well-defined and conjugate linear. On the other hand, $j(F) = j(F')$ implies easily $F=F'$.
Therefore $V_\phi(X, \mu)^\ast$ can be identified with a   closed subspace  of $\ov{V_\psi}(X, \mu):=\{\ov{\cl{\xi}{\psi}} : \xi\in \V_\psi(X, \mu)\}$, the conjugate space of $V_\psi(X, \mu)$.

Working in the framework of Hilbert spaces, as in Section \ref{subsec-hilb}, we proved in \cite{ast-reprodpairs} that the spaces $V_\phi(X, \mu)^\ast$ and $\ov{V_\psi}(X, \mu)$ can be identified.
The conclusion was
that if $(\psi,\phi)$ is a reproducing pair,  the spaces $V_\phi(X, \mu)$ and $V_\psi(X, \mu)$ are both Hilbert spaces, {conjugate dual} of each other with respect to the sesquilinear form \eqref{eq-dual}.
And if $\phi$ and $\psi$ are also  $\mu$-total, then the converse statement holds true. 

In the present situation, however, a result of this kind cannot be proved with techniques similar to those adopted in \cite{ast-reprodpairs}, which are  {specific} of Hilbert spaces. {In particular, the  condition (b), $S_{\psi,\phi}\in GL(\H)$, which was essential in the proof of 
\cite[Lemma 3.11]{ast-reprodpairs}, is now missing, and it is not clear by what regularity condition it should replaced.}

{However,   \emph{assume} that $\ran \widehat C_{\psi,\phi}[\norm{\phi}{\cdot}]=V_\phi(X,\mu)[\|\cdot\|_\phi]$
and $\ran \widehat C_{\phi,\psi}[\norm{\psi}{\cdot}]=V_\psi(X,\mu)[\|\cdot\|_\psi]$, where  we have defined the operator  $\widehat{C}_{\phi ,\psi}  :  \H \to V_\psi(X, \mu) $ by
 $\widehat{C}_{\phi ,\psi} f:= [C_\phi f]_\psi $ and  similarly for $\widehat C_{\psi,\phi}$. Then the proof of \cite[Theorem 3.14]{ast-reprodpairs} works and the same result may be obtained. This is, however, a strong and non-intuitive assumption.}

\section{Reproducing pairs and genuine \pip s}
\label{sec-PIP}

In this section, we will consider the case where our measurable functions take their values in a genuine \pip. However, for simplicity, we will restrict ourselves to a lattice of Banach spaces (LBS) or a lattice of Hilbert spaces (LHS) \cite{pip-book}.  {For the convenience of the reader, we have summarized in the Appendix the basic notions concerning LBSs and LHSs.}

Let $(X, \mu)$ be a locally compact, $\sigma$-compact measure space. Let $V_J = \{V_p, p\in J\}$ be a LBS  or a LHS of measurable functions with the property 
\be\label{defVp}
\xi \in V_p, \;\eta \in V_{\ov p} \;\Longrightarrow\; \xi\overline{\eta} \in L^1 (X, \mu)\qquad \mathrm{and}
 \qquad  \left| \int_X \xi(x) \overline{\eta (x)} \ud\mu (x)\right| \leq \|\xi\|_p \,  \|\eta\|_{\ov p}.
 \en
Thus the central \hs\ is  {$\H:=V_o=  L^2 (X, \mu)$} and the spaces $V_p, V_{\ov p}$ are dual of each other with respect to the $L^2$ inner product. The partial inner product, which extends that of  $L^2 (X, \mu)$, is denoted again by $\ip{\cdot}\cdot$.
As usual we put $V= \sum_{p\in J} V_p$ and $V^\#= \bigcap_{p\in J}V_p.$ Thus $\psi : X \to V$ means that $\psi : X \to V_p$ for some $ p\in J$.

\beex A typical example is the lattice generated by the Lebesgue spaces $L^p(\RN, \ud x), \,1\leq p \leq \infty $,
with $\frac1p +\frac{1}{\ov{p}} = 1$ \cite{pip-book}. We shall discuss it in detail in Section \ref{sec-Lp}.
\enex

We will envisage two approaches, depending whether  the functions $\psi_x$ themselves belong to $V$ or rather the scalar functions $C_\psi f$.

\subsection{Vector-valued measurable functions $\psi_x$}

This approach is the exact generalization of the one used in the RHS case.
Let $x\in X \mapsto \psi_x, \,x\in X \mapsto \phi_x$ weakly measurable functions from $X$ into $V$, 
 where the latter is equipped with the weak topology $\sigma(V,V\co)$.  More precisely, {assume that } $\psi : X \to V_p$ for some $p\in J$ and $\phi: X \to V_q$ for some $q\in J$, both weakly measurable. In that case, the analysis of Section \ref{subsec-hilb} may be repeated  \emph{verbatim}, simply replacing $\D$ by $V^\#$, thus defining 
 reproducing pairs.
The problem with this approach is that, in fact,  it does not exploit the \pip\ structure, only the RHS $V^\# \subset \H \subset V$\,!
Clearly,  this approach yields no  benefit, so we turn to a different strategy.}

\subsection{Scalar measurable functions $C_\psi f$}

Let $\psi, \phi$ be weakly measurable functions from $X$ into  {$\H$}. 
In view of \eqref{eq-defS}, \eqref{defVp} and  the definition of $V$, we assume that the following condition holds:
\begin{itemize}
\item[{\sf (p)} ]
 { $ \exists \, p\in J $ such that $ C_\psi f =\ip{f}{\psi_{\cdot}}\in V_p$ and $ C_\phi g = \ip{g}{\phi_{\cdot}}\in V_{\ov p}, \forall\, f,g \in \H $}.
\end{itemize}
We  recall that $V_{\ov p}$ is the conjugate dual of $V_{ p}$. In this case, then
$$
\Omega_{\psi,\phi}(f,g): = \int_X \ip{f}{\psi_x}\ip{\phi_x}{g}\ud\mu(x),\;  f,g\in\H,
$$
defines a sesquilinear form on  {$\H\times\H$} and one has
\be\label{eq-formpp}
|\Omega_{\psi,\phi}(f,g) | \leq \norm{p}{C_\psi f} \, \norm{\ov p}{C_\phi g}, \;\forall\, f,g\in\H.
\en
If $\Omega_{\psi,\phi}$ is bounded as a form on  {$\H\times \H$}  (this is not automatic, see Corollary \ref{cor55}),
there exists a bounded operator $S_{\psi,\phi}$ in  {$\H$} such that
\begin{equation}\label{eqn_*}
\int_X \ip{f}{\psi_x}\ip{\phi_x}{g}\ud\mu(x)= \ip{S_{\psi,\phi}f}{g}, \;\forall \,f,g \in  {\H}.
\end{equation}
Then  $(\psi,\phi)$ is a \emph{reproducing pair} if $S_{\psi,\phi}\in GL( {\H})$.
\medskip

Let us suppose that the spaces $V_p$ have the following property
\begin{itemize}
\item[\sf{(k)}] If $\xi_n \to \xi$ in $V_p$, then, for every compact subset $K \subset X$, there exists a subsequence $\{\xi^K_n\}$ of $\{\xi_n\}$ which converges to $\xi$ almost everywhere in $K$.
    \end{itemize}
We note that condition {\sf{(k)}} is satisfied by $L^p$-spaces \cite{rudin}.

As seen before, $C_\psi: \H\to V$, in general. This means, given $f\in \H$,  there exists $p \in J$  such that $C_\psi f =  \ip{f}{\psi_{\cdot}}\in V_p$. We define
$$
D_r(C_\psi)= \{f \in \H:\, C_\psi f \in V_r\}, \; r\in J.
$$
 {In particular, $D_r(C_\psi)= {\H}$ means $C_\psi (\H) \subset V_r$.}

\begin{prop} Assume that {\sf{(k)}} holds. Then $C_\psi: D_r(C_\psi) \to V_r$ is a closed linear map.
\end{prop}
\begin{proof}
 Let $f_n \to f$ in  {$\H$} and $\{C_\psi f_n\}$ be Cauchy in $V_r$. Since $V_r$ is complete, there exists $\xi \in V_r$ such that $\|C_\psi f_n-\xi\|_r\to 0$.
 By {\sf{(k)}}, for every compact subset $K \subset X$, there exists a subsequence $\{f_n^K\}$ of $\{f_n\}$ such that $(C_\psi f_n^K)(x) \to \xi(x)$ a.e. in $K$.
 On the other hand, since $f_n \to f$ in  {$\H$}, we get
 $$
 \ip{f_n}{\psi_x} \to  \ip{f}{\psi_x}, \quad \forall x \in X ,
 $$
 and the same holds true, of course, for $\{f_n^K\}$. From this we conclude that $\xi(x)=
\ip{f}{\psi_x}$ almost everywhere. Thus, $f \in D_r(C_\psi)$ and $\xi= C_\psi f$.
\end{proof}
\medskip

By a simple application of the closed graph theorem we obtain
\begin{cor}\label{cor54}
 {Assume that {\sf{(k)}} holds}. If for some $r\in J $, $C_\psi (\H) \subset V_r $, then $C_\psi:  {\H}\to V_r$ is continuous.
\end{cor}
Combining Corollary \ref{cor54} with \eqref{eq-formpp}, we get
\begin{cor}\label{cor55}
 {Assume that {\sf{(k)}} holds}. If $C_\psi (\H) \subset V_p $ and $C_\psi (\H) \subset V_{\ov p}$, the form $\Omega$ is bounded on  {$\H\times \H$}, that is,
 {$|\Omega_{\psi,\phi}(f,g) | \leq c\norm{}{ f} \, \norm{}{ g}$}.
\end{cor}

 {Hence, if condition {\sf{(k)}} holds, $C_\psi (\H) \subset V_r $ implies that $C_\psi:  {\H}\to V_r$ is continuous.
If we don't know whether the condition holds, we will have to assume explicitly that $C_\psi:  {\H}\to V_r$ is continuous.}

If $C_\psi:  {\H}\to V_r$ continuously, then $C_\psi^\ast:V_{\ov r} \to  {\H}$ exists and it is continuous. By definition, if $\xi \in V_{\ov r} $,
\be\label{eq:cont}
\ip{C_\psi f}{\xi}= \int_X  \ip{f}{\psi_x} \overline{\xi(x)}\ud\mu(x),\;\forall\,f\in\H.
\en
{The relation \eqref{eq:cont} then  implies that}
$$
 \int_X  \ip{f}{\psi_x} \overline{\xi(x)}\ud\mu(x)=   \ip{f}{\int_X \psi_x {\xi(x)}\ud\mu(x)},\;\forall\,f\in\H.
$$
Thus,
$$
 C^*_\psi\xi = \int_X \psi_x {\xi(x)}\ud\mu(x).
 $$

Of course, what we have said about $C_\psi$ holds in the very same way for $C_\phi$. Assume now that for some
 $p \in J, C_\psi:  {\H}\to V_p$ and $C_\phi:  {\H}\to V_{\ov p}$ continuously.
 Then, $C_\phi^*:V_p \to  {\H}$ so that
 $C_\phi^*C_\psi$ is a well-defined bounded operator in  {$\H$}. As before, we have
$$
 C_\phi^* \eta = \int_X \eta(x) \phi_x \ud\mu(x), \; \forall \, \eta\in V_p.
 $$
Hence,
$$
 C_\phi^*C_\psi f = \int_X \ip{f}{\psi_x}\phi_x\ud\mu(x) = S_{\psi, \phi} f, \quad \forall\, f\in  {\H},
 $$
the last equality following also {from \eqref{eqn_*} and Corollary \ref{cor55}.} Of course,  this does not yet imply that  $S_{\psi, \phi}\in GL( {\H})$, thus we don't know whether $(\psi, \phi)$ is a reproducing pair.

 {Let us now return to the pre-Hilbert space $\V_\phi(X, \mu) $.  First,  the defining relation (3.3) of  \cite{ast-reprodpairs} must be written as
 $$
\xi \in \V_\phi(X, \mu)  \; \Leftrightarrow \;  \left|  \int_X  \xi(x)  \ov{(C_\phi g )(x)}  \ud\mu(x) \right| \leq c \norm{}{g}, \; \forall\, g \in \H.
 $$
Since $C_\phi :\H \to V_{\ov p}$, the integral  {is well defined} for all $\xi \in V_p$. This means, the inner product  on the r.h.s.  
is in fact  the partial inner product of $V$, which coincides with the $L^2$ inner product whenever the latter makes sense.
We may rewrite the r.h.s.  as
$$
|\ip{\xi}{C_\phi g}| \leq c \norm{}{g}, \forall\, g \in  {\H},\;\xi \in V_p.  
$$
where $\ip{\cdot}{\cdot}$ denotes the partial inner product. Next, by \eqref{defVp}, one has, for $\xi \in V_p, g\in\H$,
$$
 |\ip{\xi}{C_\phi g}| \leq \norm{p}{\xi} \, \norm{\ov p}{C_\phi g} \leq c \norm{p}{\xi} \norm{ }{g},
 $$
 where the last inequality follows from Corollary \ref{cor54} {or the assumption of continuity of $C_\phi$}. Hence  indeed $\xi \in \V_\phi(X, \mu)$, so that $V_p \subset \V_\phi(X, \mu)$.}

As for the adjoint operator, we have $C_\phi^* : V_p \to \H$. Then we may write, for $\xi\in V_p, g\in\H$,
$\ip{\xi}{C_\phi g}  = \ip{ T_\phi\xi}{g}$, thus  {$C_\phi^*$ is the restriction from $\V_\phi(X, \mu)$ to  $V_p$ of the operator
{$T_\phi: \V_\phi \to \H$ introduced in Section \ref{sec-prel}, which reads now as
\be\label{eq-Tphi-p}
\ip{T_\phi   \xi}{g}    =\int_X  \xi(x)  \ip{\phi_x}{g} \ud\mu(x) , \; { \forall} \, \xi \in V_p, g\in\H.
\en
Thus $C_\phi^* \subset T_\phi$.}

Next, the construction proceeds as in Section \ref{sec-RHS}.  The space $ V_\phi(X, \mu)= \V_\phi(X, \mu)/{\Ker}\,T_\phi$, with the norm
 $\norm{\phi}{\cl{\xi}{\phi}}=\norm{}{T_\phi \xi}$, is a pre-Hilbert space. Then Theorem 3.14 and the other results from Section 3 of \cite{ast-reprodpairs} remain true. In particular, we have:
  \betheo\label{theo-dual2}
If $(\psi,\phi)$ is a reproducing pair,  the spaces $V_\phi(X, \mu)$ and $V_\psi(X, \mu)$ are both Hilbert spaces, {conjugate dual} of each other with respect to the sesquilinear form \eqref{eq_sesq}, namely, 
$$
\ipp{\xi}{\eta}:= \int_X \xi(x)  \ov{\eta(x) } \ud\mu(x).
$$
\entheo
Note the form \eqref{eq_sesq} coincides with the inner product of $L^2(X, \mu)$ whenever the latter makes sense.
\medskip

{Let $(\psi,\phi)$ is a reproducing pair.}
 {Assume again  that    $C_\phi :  {\H} \to V_{\ov{ p}}$ continuously, which me may write
 $\widehat{C}_{\phi ,\psi}  :   {\H} \to V_{\ov{ p}}/{{\Ker}\,T_\psi}$, where  $\widehat{C}_{\phi ,\psi}  :  \H \to V_\psi(X, \mu) $ is  the operator  defined  by
 $\widehat{C}_{\phi ,\psi} f:= [C_\phi f]_\psi $, already introduced at the end of Section \ref{sec-alt}.   In addition, by \cite[Theorem 3.13]{ast-reprodpairs}, one has $\ran \widehat C_{\psi,\phi}[\norm{\phi}{\cdot}]=V_\phi(X,\mu)[\|\cdot\|_\phi]$
and $\ran \widehat C_{\phi,\psi}[\norm{\psi}{\cdot}]=V_\psi(X,\mu)[\|\cdot\|_\psi]$.

 Putting everything together, we get
\begin{cor}\label{cor56}
Let $(\psi, \phi)$ be a reproducing pair. Then,  {if $C_\psi :  {\H} \to V_{ p}$ and $C_\phi :  {\H} \to V_{\ov{ p}}$ continuously,}
one has
\begin{align}
\widehat{C}_{\phi ,\psi} :  {\H} \to V_{\ov{ p}}/{{\Ker}\,T_\psi} =  V_\psi(X, \mu) \simeq {\ov{V_\phi}}(X, \mu)^* , \label{eq-Cphipsi}
\\
\widehat{C}_{\psi ,\phi} :  {\H} \to V_{ p}/{\Ker}\,T_\phi  = V_\phi(X, \mu) \simeq   {\ov{V_\psi}}(X, \mu)^* . \label{eq-Cpsiphi}
\end{align}
 In these relations, the equality sign   means an isomorphism of vector spaces, whereas $\simeq$ denotes an isomorphism of Hilbert spaces.
 \encor
{ \begin{proof}
 On one hand, we have $\ran \widehat C_{\phi,\psi} =V_\psi(X,\mu) $. On the other hand, under the assumption $C_\phi( {\H})\subset V_{\ov p}$, one has
 $V_{\ov p} \subset \V_\psi(X, \mu)$, hence $V_{\ov p}/{{\Ker}\,T_\psi}=\{\xi +\Ker T_\psi, \; \xi \in  V_{\ov p}\}
  \subset V_\psi(X, \mu)$.   
 Thus we get $V_\psi(X, \mu) = V_{\ov p}/{{\Ker}\,T_\psi}$ as vector spaces.
Similarly   $V_\phi(X, \mu) =  V_{p}/{{\Ker}\,T_\phi}.$
 \medskip
  \end{proof}}

{Notice that, in Condition {\sf (p)},  the index $p$ cannot depend  on $f,g$.}
 We need some uniformity,  in the form $C_\psi( {\H})\subset V_p$ and $C_\phi( {\H})\subset V_{\ov p}$ . This is fully in line with the
 philosophy of \pip s: the building blocks are the (assaying) subspaces $V_p$, not individual vectors.}

 {\section{The case of a Hilbert triplet or a Hilbert scale}}
\label{sec-Hchain}

\subsection{The general construction}
\label{subsec-genconst}

We have derived in the previous section the relations $V_p \subset \V_\phi(X, \mu), V_{\ov{ p}}\subset \V_\psi(X, \mu)$,  and their equivalent ones \eqref{eq-Cphipsi}, \eqref{eq-Cpsiphi}. Then, since $V_\psi(X, \mu) $ and $V_\phi(X, \mu) $ are both Hilbert spaces, it seems natural to take for $V_p, V_{\ov{ p}}$ 
 Hilbert spaces as well, that is, take for $V$ a LHS. The simplest case is then a Hilbert chain, for instance, the scale \eqref{eq:scaleHn} $\{\H_k, k\in\ZN\}$ built on the powers of a self-adjoint operator $A>I$ . This situation is quite interesting, since in that case one may get results about spectral properties of symmetric operators (in the sense of \pip\ operators) \cite{at-pipops}.
 
 Thus, let $(\psi,\phi)$ be a reproducing pair. For simplicity, we assume that $S_{\psi,\phi}= I$, that is, $\psi,\phi$ are dual to each other.
 
 If  $\psi$ and $\phi$ are both frames, there is nothing to say, since then $C_\psi( \H),  C_\phi( \H)\subset L^2(X, \mu)=\H_o$, so that there is no need for  a Hilbert scale.   {Thus we assume that $\psi$ is an upper semi-frame and $\phi$ is a lower semi-frame, dual to each other. It follows that $C_\psi( \H)\subset L^2(X, \mu)$. Hence  Condition {\sf (p)} becomes: There is an index $k\geq 1$ such that  {if $C_\psi :  {\H} \to \H_k$ and $C_\phi :  {\H} \to \H_{\ov k}$ continuously,}
 thus $V_{p}\equiv \H_k$  and  $V_{\ov p}\equiv \H_{\ov k}$. This
 means we are working in the Hilbert triplet
 \be\label{eq-triplet}
V_{p}\equiv \H_k \subset \H_o =  L^2(X, \mu) \subset \H_{\ov k}\equiv  V_{\ov p}\,.
 \en
 Next, according to Corollary \ref{cor56}, we have  $V_\psi(X, \mu) = \H_{\ov k}/{{\Ker}\,T_\psi}$ and $V_\phi(X, \mu) =   \H_{k}/{{\Ker}\,T_\phi}$, as vector spaces.}

 {In addition, since $\phi$ is a lower semi-frame,  \cite[Lemma 2.1]{ant-bal-semiframes1} tells us that $C_\phi$ has closed range in $L^2(X, \mu)$ and is injective. However its domain 
 $$
D(C_{\phi}):= \{f\in\H : \int_{X}  |\ip{f}{\phi_{x}}| ^2 \, \ud \nu(x) <\infty\}
$$
need not be dense, it could be $\{0\}$. Thus $C_\phi$ maps its domain $D(C_{\phi})$ onto a closed subspace of  $L^2(X, \mu)$, possibly trivial, and the whole of $\H$ into the larger space $\H_{\ov k}$.  
}

\subsection{Examples}
\label{subsec-ex}

 {As for concrete examples of such Hilbert scales, we might mention two. First the Sobolev spaces $H^k(\RN), \, k\in \ZN$, in  
$\H_0 = L^2(\RN, dx)$, which is the scale generated by the powers of the self-adjoint operator $A^{1/2}$, where $A  := 1 -\frac{\ud^2}{\ud x^2}$.
The other one corresponds to the quantum harmonic oscillator, with Hamiltonian  $A_\mathrm{osc} := x^2-\frac{\ud^2}{\ud x^2}$. The spectrum of  $A_\mathrm{osc}$ is $\{2n+1, n= 0,1,2,\ldots\}$ and it gets diagonalized on the basis of Hermite functions. 
It follows that $A_\mathrm{osc}^{-1}$, which maps every  $\H_k$ onto  $\H_{k-1}$, is a Hilbert-Schmidt operator. Therefore, the end space of the scale $\D^{\infty}(A_\mathrm{osc}):=\bigcap_{k} \H_k$, which is simply Schwartz' space $\mathcal S$ of $C^\infty$ functions of fast decrease, is a nuclear space.
}

{Actually one may give an explicit example, using a Sobolev-type scale. Let $\H_K$ be a reproducing kernel Hilbert space (RKHS) of (nice) functions on 
ameasure space $(X, \mu)$, with kernel function $k_x, x\in X$, that is, $f(x)=\ip{f}{k_x}_K,\, \forall f\in\H_{K}$. The corresponding reproducing kernel is $K(x,y)=k_y(x)=\ip{k_y}{k_x}_K$. Choose the weight  function $m(x) >1$, the analog of the weight $(1+|x|^2)$ considered in the Sobolev case. Define the Hilbert scale $\H_k, \, k\in \ZN$, determined by the multiplication operator $Af(x) = m(x) f(x), \, \forall x\in X$.
Hence, for each $l \geq 1$,
$$
 \H_l \subset\H_0 \equiv  \H_K \subset \H_{\ov l}\,. 
$$
Then, for some $n \geq 1$, define the measurable functions $\phi_x = k_x m^n(x), \psi_x = k_x m^{-n}(x)$, so that 
 {$C_\psi : \H_K \to   \H_n,  \, C_\phi : \H_K \to  \H_{\ov n}$ continuously, }and they are dual of each other. Thus $(\psi, \phi)$ is a reproducing pair, with $\psi$ an upper semi-frame and $\phi$ a lower semi-frame.
}

{In this case, one can compute the operators $T_\psi, T_\phi$ explicitly. The definition \eqref{eq-Tphi-p} reads as
$$
\ip{T_\phi   \xi}{g}_K    =\int_X  \xi(x)  \ip{\phi_x}{g}_K \ud\mu(x) , \; { \forall} \, \xi \in \H_n, g\in\H_K.
$$
Now $ \ip{\phi_x}{g}_K = \ip{k_x m^n(x)}{g}_K =  \ip{k_x}{g \, m^n(x)}_K = \ov{g(x)}\,m^n(x) \in \H_{\ov n}$.
Thus
$$
\ip{T_\phi   \xi}{g}_K  = \int_X  \xi(x) \, \ov{g(x)}\,m^n(x) \ud\mu,
$$
that is, $(T_\phi   \xi )(x) = \xi(x)\, m^n(x) $ or $T_\phi   \xi = \xi \,m^n$.
However, since the weight $m(x)>1$ is invertible, $\ov{g}\,m^n$ runs over the whole of $\H_{\ov n}$ whenever $g$ runs over $H_K$.
Hence  $\xi\in \Ker\,T_\phi \subset \H_n$ means that $\ip{T_\phi   \xi}{g}_K =0, \, \forall g\in H_K$, which implies $\xi = 0$, since the duality between
$\H_n$ and $\H_{\ov n}$ is separating. The same reasoning yields $\Ker\,T_\psi = \{0\}$. Therefore $V_\phi(X, \mu) = \H_n$ and $V_\psi(X, \mu) = \H_{\ov n}$.
}

   \medskip
   
 {A more general situation may be derived from }the discrete example of Section 6.1.3 of \cite{ast-reprodpairs}. Take a sequence of weights $m:=\{|m_n|\}_{n\in\NN}\in c_0, m_n\neq 0,$ and consider the space $\ell^2_m$ with norm $\norm{\ell^2_m}{\xi}:=\sum_{n\in\NN}|m_n \xi_n |^2$. Then we have the following triplet replacing \eqref{eq-triplet}
  \be\label{eq-triplet2}
 \ell^2_{1/m} \subset \ell^2  \subset \ell^2_m.
 \en 
 {Next, for each $n\in\NN$,  define $\psi_n = m_n \theta_n$, where $\theta$ is a frame or an orthonormal basis in $\ell^2$. Then $\psi$  is an upper semi-frame.
  Moreover, 
   $\phi:=\{(1/\ov{m_n})\theta_n\}_{n\in\NN} $ is a lower semi-frame, dual to $\psi$, thus $(\psi,\phi)$ is a reproducing pair.
 Hence, by \cite[Theorem 3.13]{ast-reprodpairs},  $V_\psi \simeq\ran C_\phi = M_{1/m} (V_\theta(\NN)) =\ell^2_m$ and
$V_\phi \simeq\ran C_\psi = M_{m} (V_\theta(\NN)) =\ell^2_{1/m}$
(here we take for granted that ${\Ker}\,T_\psi = {\Ker}\,T_\phi = \{0\}$).}

 {For making contact with the situation of \eqref{eq-triplet}, consider in $\ell^2 $ the diagonal operator $A:= \mathrm{diag}[n], n\in \NN$, that is $(A\xi)_n = n \,\xi_n, n\in \NN$, which is obviously self-adjoint and larger than 1. Then $ \H_k = D(A^k)$ with norm  $\norm{k}{\xi}=  \norm{}{A^k\xi}\equiv \ell^2_{r^{(k)}}$, where $(r^{(k)})_n = n^k$ (note that $1/{r^{(k)}}\in c_0$).
Hence we have
 \be\label{eq-triplet3}
 \H_k= \ell^2_{r^{(k)}} \subset \H_o = \ell^2 \subset \H_{\ov k}=\ell^2_{1/{r^{(k)}}}\,,
\en
 where $(1/r^{(k)})_n = n^{-k}$. In addition, as in the continuous case discussed above, the end space of the scale,  $\D^{\infty}(A):=\bigcap_{k} \H_k$, is simply Schwartz' space $s$ of fast decreasing sequences, with dual $ \D_{\ov \infty}(A):=\bigcup_{k} \H_{k} = s'$, the space of slowly increasing sequences. Here too, 
  this construction shows  that the space $s$ is nuclear, since every embedding $A^{-1} : \H_{k+1} \to \H_k$ is a Hilbert-Schmidt operator.}
  
However, the construction described above yields a much more general family of examples, since the weight sequences $m$ are not ordered.}

\section{The case of $L^p$ spaces}
\label{sec-Lp}

Following the suggestion made at the  {end} of Section \ref{sec-prel}, we present now several possibilities of taking  $\ran C_\psi$ in the context of the Lebesgue spaces $L^p(\RN,\ud x)$.

As it is well-known, these spaces  don't form a chain, since two of them are never  comparable. We have only
$$
 L^p \cap L^q \subset L^s, \, \mbox {for all }     s  \mbox{ such that } \; p<s<q.
$$
Take the lattice $\J$ generated by
$\I= \{ L^p(\RN,\ud x), 1 \leq p \leq \infty\}$, with lattice operations  \cite[Sec.4.1.2]{pip-book}:
\bei
\item
  $L^p \wedge L^q = L^p \cap L^q$ is a Banach space for the projective   norm
$\| f \|_{p \wedge q} = \| f \|_p + \| f \|_q$
 \item
  $L^p \vee L^q = L^p + L^q$ is a Banach space  for the inductive norm
 \\
\hspace*{2cm}$\| f \|_{p \vee q} = \inf_{f=g+h} \left\{\| g \|_p + \| h \|_q;  g \in L^p, \, h \in L^q\right\}$
\item
 For $1<p,q<\infty$, both spaces $L^p \wedge L^q$ and $L^p \vee L^q$ are reflexive and
$(L^p \wedge L^q)^\times = L^{\ov p} \vee L^{\ov q}$.
 \eni
Moreover, no additional spaces are obtained by iterating the lattice operations to any finite order. Thus we obtain an involutive lattice and a LBS, denoted by $\VJ $.

It is convenient to introduce a unified notation:
$$
L^{(p,q)} = \left\{ \begin{array}{ll} L^p \wedge L^q = L^p \cap L^q, & \mbox{ if } \; p \geq q, \\
 L^p\vee L^q =L^p + L^q, & \mbox{ if } \; p \leq q.
\end{array}
\right.
$$
Following  \cite[Sec.4.1.2]{pip-book}, we represent  the   space $ L^{(p,q)}$ by the point $(1/p,1/q)$ of the  unit square ${\rm J} = [0,1]\times [0,1]$.
In this representation,  the spaces $L^p$ are on the main diagonal, intersections $L^p \cap L^q$ above it and sums $L^p + L^q$ below, the duality is $[L^{(p,q)}]^\times = L^{(\ov{p},\ov{q})}$, that is,  symmetry with respect to $L^2$.
Hence,  $L^{(p,q)}\subset L^{(p',q')}$ if $(1/p,1/q)$ is on the left and/or above $(1/p',1/q')$
The extreme spaces are
 \footnote{The space $L^1 + L^\infty$ has been considered by  Gould  \cite{gould}.}
$$
\VJ\co
= L^{(\infty,1)} = L^{\infty} \cap L^{1},
\quad \mbox{ and } \quad \VJ
= L^{(1,\infty)} = L^{1} + L^{\infty}.
$$
For a full picture, see \cite[Fig.4.1]{pip-book}.

There are three possibilities for using the $L^p$ lattice for controlling reproducing  {pairs}

(1) Exploit the  \emph{full lattice $\J$}, that is,    find $(p,q)$ such that,  $\forall f,g\in\H$,
 $ C_\psi f \,  \# \, C_\phi g $ in the \pip\ $\VJ $, that is,
$ C_\psi f \in  L^{(p,q)}$    and  $C_\phi g \in L^{(\ov{p},\ov{q})}$.

(2) Select in $\VJ$ a self-dual \emph{Banach chain} $\VI$, centered around $L^2$,
symbolically.
\be\label{eq:chain}
\ldots L^{(s)} \subset \ldots \subset L^2  \subset \ldots \subset L^{(\ov s)} \ldots ,
\en
such that $ C_\psi f \in  L^{(s)}$    and  $C_\phi g \in L^{(\ov{s})}$ (or vice-versa).
 Here are three examples of such Banach chains.

\begin{figure}[t]
\centering

\setlength{\unitlength}{1cm}

\begin{picture}(7,11)

\put(-14,1){
\begin{picture}(7,0)

\put(12.5,0){\vector(0,1){9}}
\put(12.2,9.3){\shortstack{$ 1/q$}}

\put(12.5,0){\vector(1,0){9}}
\put(22.2,0){\makebox(0,0){$  1/p$}}

\put(12.5,0){\line(1,1){8}}
\put(20.5,0){\line(0,1){8}}
\put(12.5,8){\line(1,0){8}}

 \put(14.5,2){\color{red}\dashbox{0.1}(4,4)}
  \put(12.5,0){\color{red}\dashline [+40]{0.12}(0,6)(2,6)}
  \put(12.5,0){\color{red}\dashline [+40]{0.12}(0,2)(2,2)}

\put(12.5,8){\line(1,-1){8}}
\put(14.5,5){\color{blue}\line(2,-1){4}}
\put(15.5,2){\color{blue}\line(1,2){2}}

\put(12.5,4){\dashbox{0.1}(8,0)}
\put(16.5,0){\dashbox{0.1}(0,8)}

\put(12,-0.5){\makebox(0,0){$  L^{\infty} $}}
\put(21,8.5){\makebox(0,0){$  L^{1} $}}

\put(20.6,-0.5){\makebox(0,0){$  L^{1} + L^{\infty}$}}
\put(16.6,-0.5){\makebox(0,0){$ L^{2} + L^{\infty}$}}
\put(11.5,8.2){\makebox(0,0){$ L^{\infty} \cap L^{1}$}}
\put(11.5,6.2){\makebox(0,0){$ L^{\infty} \cap L^{(t)}$}}
\put(11.5,2.2){\makebox(0,0){$ L^{\infty} \cap L^{(\ov t)}$}}
\put(11.5,4.2){\makebox(0,0){$L^{\infty} \cap L^{2}$}}
\put(21.5,4.2){\makebox(0,0){$L^{1} + L^{2}$}}
\put(16.6,8.5){\makebox(0,0){$  L^{2} \cap L^{1}$}}

\put(16.7,3.45){\makebox(0,0){$ L^2$}}
\put(18.5,6.4){\makebox(0,0){\color{red}$  L^{q} $}}
\put(14.1,2.3){\makebox(0,0){\color{red}$  L^{\ov q} $}}
\put(15.2,6.4){\makebox(0,0){$\color{red} L^{\ov q} \cap L^q $}}
\put(18.9,1.6){\makebox(0,0){$\color{red}  L^{q}+ L^{\ov q} = (L^{\ov q} \cap L^q)^\times $}}
\put(14,5.1){\makebox(0,0){$  L^{(s)} $}}
\put(19.1,3.1){\makebox(0,0){$  L^{(\ov{s})} $}}

\put(17.5,6.4){\makebox(0,0){$  L^{(t)} $}}
\put(15.5,1.6){\makebox(0,0){$  L^{(\ov{t})} $}}

\put(12.5,2){\makebox(0,0){$\scriptstyle\bullet$}}
\put(12.5,6){\makebox(0,0){$\scriptstyle\bullet$}}

\put(12.5,0){\makebox(0,0){$\scriptstyle\bullet$}}
\put(12.5,4){\makebox(0,0){$\scriptstyle\bullet$}}
\put(12.5,8){\makebox(0,0){$\scriptstyle\bullet$}}

\put(14.5,2){\makebox(0,0){$\scriptstyle\bullet$}}
 \put(14.5,6){\makebox(0,0){$\scriptstyle\bullet$}}
 \put(14.5,5){\makebox(0,0){$\scriptstyle\bullet$}}

\put(16.5,0){\makebox(0,0){$\scriptstyle\bullet$}}
\put(16.5,4){\makebox(0,0){$\scriptstyle\bullet$}}
\put(16.5,8){\makebox(0,0){$\scriptstyle\bullet$}}

\put(15.5,2){\makebox(0,0){$\scriptstyle\bullet$}}
\put(17.5,6){\makebox(0,0){$\scriptstyle\bullet$}}

\put(18.5,2){\makebox(0,0){$\scriptstyle\bullet$}}
\put(18.5,3){\makebox(0,0){$\scriptstyle\bullet$}}
\put(18.5,6){\makebox(0,0){$\scriptstyle\bullet$}}

\put(20.5,0){\makebox(0,0){$ \scriptstyle\bullet$}}
\put(20.5,4){\makebox(0,0){$\scriptstyle\bullet$}}
\put(20.5,8){\makebox(0,0){$\scriptstyle\bullet$}}

\end{picture}}

\end{picture}
\caption{(i) The pair  $L^{(s)},L^{(\ov s)}$ for $s$ in the second quadrant;
(ii) The pair  $L^{(t)},L^{(\ov t)}$ for $t$ in the first quadrant.}
\label{fig1}
\end{figure}
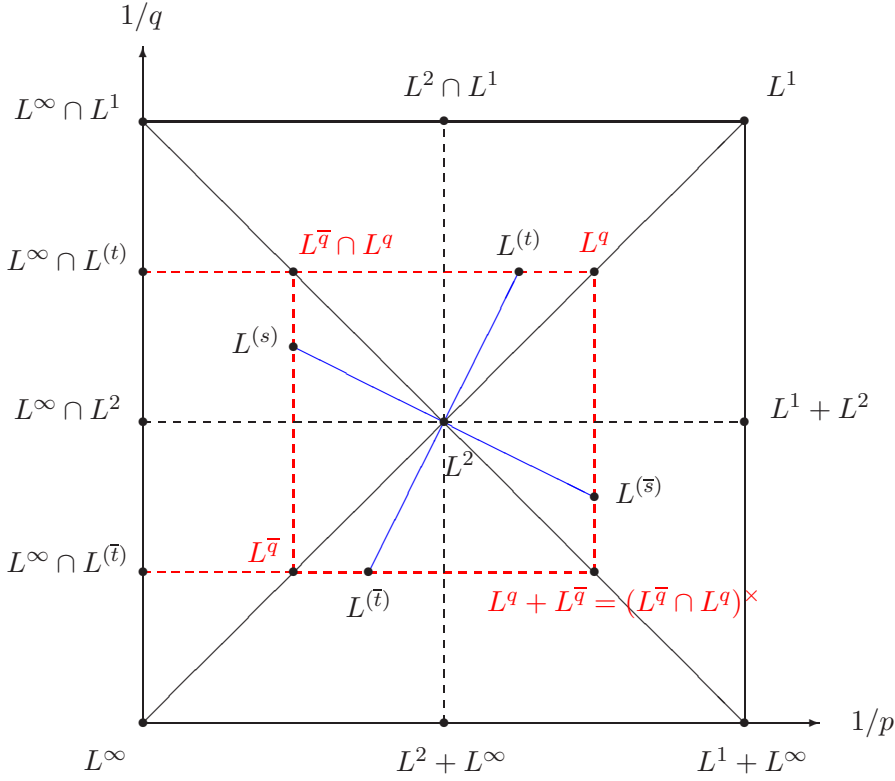

 \begin{itemize}
  \item    The \emph{diagonal} chain : $q= \ov{p}$
  $$
\hspace*{-8mm}  L^{\infty} \cap L^{1}\subset  \ldots \subset   L^{\ov q} \cap L^q  \subset  \ldots \subset
   L^2  \subset  \ldots \subset L^{q}+ L^{\ov q}
  = (L^{\ov q} \cap L^q)^\times  \subset  \ldots \subset L^{1} + L^{\infty}.
  $$
    \item    The horizontal  chain    $q=2$ :
   $$
\hspace*{-8mm}  L^{\infty} \cap L^{2}\subset  \ldots   \subset   L^2  \subset   \ldots \subset L^{1} + L^{2}.
  $$
 \item    The vertical chain $p=2$ :
  $$
   L^{2} \cap L^{1}\subset  \ldots     \subset L^2  \subset   \ldots \subset  L^{2} + L^{\infty}.
  $$
\eni
All three chains are presented in Figure \ref{fig1}.
In this case, the full chain belongs to the second and fourth quadrants (top left and bottom right). A typical point is then 
 $s=(p,q)$ with, $2\leq p \leq \infty, 1\leq q \leq 2$, so that one has
  the situation depicted in \eqref{eq:chain}, that is, the spaces $L^{(s)}, L^{(\ov s)}$ to which
$C_\psi f$, resp. $C_\phi g$,  belong, are necessarily comparable to each other and to $L^2$. In particular, one of them is necessarily contained in $L^2$.
Note the extreme  spaces of that type are $L^2, L^{\infty} \cap L^{2}, L^{\infty} \cap L^{1}$ and $  L^{2} \cap L^{1}$ (see Figure \ref{fig1}).
 
(3) Choose a dual pair in the  first and third quadrant (top right, bottom left). A typical point is then $t=(p',q')$, with $1< p', q' < 2$, so that the spaces $L^{(t)}, L^{(\ov t)}$ are never comparable to each other, nor to $L^2$.   

Let us now add the uniform boundedness condition mentioned at the end of Section \ref{sec-prel}, $\sup_{x\in X}\norm{\H}{\psi_x } \leq c$  and $\sup_{x\in X}\norm{\H}{\phi_x } \leq c'$ for some $c, c'>0$.  
Then $C_\psi f(x) = \ip{f}{\psi_x} \in L^\infty(X,\ud \mu)$ and $C_\phi f(x) = \ip{f}{\phi_x} \in L^\infty(X,\ud \mu)$.
Therefore, the third case reduces to the second one, since we have now (in the situation of Figure \ref{fig1}).
$$
\hspace*{-8mm}  L^{\infty} \cap L^{(t)}\subset   L^{\infty} \cap L^2  \subset    L^{\infty} \cap L^{(\ov t)}.
$$

 {Following the pattern of Hilbert scales, we choose a (Gel'fand) triplet of Banach spaces.
One could have, for instance, a triplet of reflexive Banach spaces such as
\be\label{eq:triplet(s)}
 L^{(s)} \subset \ldots \subset L^2  \subset \ldots \subset L^{(\ov s)},
 \en
corresponding to a point $s$ inside of the second quadrant, as shown in Figure \ref{fig1}.  In this case, according to \eqref{eq-Cphipsi} and \eqref{eq-Cpsiphi}, } $V_\psi = L^{(\ov s)}/\Ker T_\psi $ and $ V_\phi = L^{(s)}/\Ker T_\phi$. 

 {On the contrary, if we choose a point $t$ in the second quadrant, case (3) above, it seems that no triplet arises. However, if $(\psi, \phi)$ is a nontrivial reproducing pair, with  $S_{\phi, \psi}=I$, that is, $\psi, \phi$ are dual to each other, one of them, say $\psi$, is an upper semi-frame and then necessarily $\phi$  is a lower semi-frame 
 \cite[Prop.2.6]{ant-bal-semiframes1}. Therefore
$C_\psi( \H)\subset L^2(X, \mu)$, that is, case (3) cannot be realized. }
\indent  {Inserting the boundedness condition, we get a triplet where the extreme spaces are no longer reflexive, such as  
$$
 L^{\infty} \cap L^{(t)}\subset   L^{\infty} \cap L^2  \subset    L^{\infty} \cap L^{(\ov t)},
 $$
and then $V_\psi = (L^{\infty} \cap L^{(t)})/\Ker T_\psi $ and $ V_\phi = (L^{\infty} \cap L^{(\ov t)})/\Ker T_\phi$.
}

 In conclusion, the only acceptable solution is the triplet \eqref{eq:triplet(s)}, with $s$ strictly inside of the second quadrant, that is, $s=(p,q)$ with, $2\leq p <\infty, 1< q \leq 2$.

A word of explanation is in order here, concerning the relations $V_\psi = L^{(\ov s)}/\Ker T_\psi $ and $ V_\phi = L^{(s)}/\Ker T_\phi$. On the l.h.s., $L^{(s)}$ and $ L^{(\ov s)}$ are reflexive Banach spaces, with their usual norm, and so are the quotients by $T_\psi $, resp. $T_\phi $.
On the other hand, $V_\psi(X,\mu)[\|\cdot\|_\psi]$ and $V_\phi(X,\mu)[\|\cdot\|_\phi]$ are Hilbert spaces.
However, there is no contradiction, since the equality sign $=$ denotes an isomorphism of vector spaces only, without reference to any topology. 
Moreover, the two norms, Banach and Hilbert, \emph{cannot} be comparable, lest they are equivalent \cite[Coroll. 1.6.8]{megg}, which is impossible in the case of $L^p, p\neq 2$. The same is true for any LBS where the spaces $V_p$ are not Hilbert spaces.

Although we don't have an explicit example of a reproducing pair, we  indicate a possible construction towards one. Let $\theta^{(1)}: \RN \to L^2$ be a measurable function such that $\ip{h}{\theta^{(1)}_x}\in L^q, \, \forall\, h\in L^2, 1<q<2$ and let $\theta^{(2)}: \RN \to L^2$ be a measurable function such that $\ip{h}{\theta^{(2)}_x}\in L^{\ov q}, \, \forall\, h\in L^2$. Define $\psi_x := \min(\theta^{(1)}_x, \theta^{(2)}_x) \equiv \theta^{(1)}_x\wedge \theta^{(2)}_x$ and 
$\phi_x := \max(\theta^{(1)}_x, \theta^{(2)}_x) \equiv \theta^{(1)}_x\vee \theta^{(2)}_x$ . Then we have
\begin{align*}
(C_\psi h)(x) = \ip{h}{\psi_x}\in L^q \cap  L^{\ov q}, \, \forall\, h\in L^2
\\
(C_\phi h)(x) = \ip{h}{\phi_x}\in L^q +  L^{\ov q}, \, \forall\, h\in L^2
\end{align*}
and we have indeed $L^q \cap  L^{\ov q} \subset L^2 \subset L^q +  L^{\ov q}$.
It remains to guarantee that $\psi$ and $\phi$ are dual to each other, that is,
$$
\int_X \ip{f}{\psi_x}\ip{\phi_x}{g}\ud\mu(x) =  \int_X C_\psi f (x) \; \ov{C_\phi g(x)}  \ud\mu(x) = \ip{f}{g}, \, \forall \, f,g\in L^2.
$$

\section{Outcome}

 We have seen in \cite{ast-reprodpairs} that the notion of reproducing pair is quite rich. It generates a whole mathematical structure, which ultimately leads
 to a pair of Hilbert spaces, conjugate dual to each other with respect to the $ L^2(X,\mu)$ inner product.  This suggests that  one should make more precise the best assumptions   on the measurable functions or, more precisely, on the nature of the range of the
analysis operators  $C_\psi, C_\phi$. This in turn suggests to analyze the whole structure in the language of \pip s, which is the topic of the present paper.
 {In particular, a natural choice is a  scale, or simply a triplet, of Hilbert spaces, the two extreme spaces being conjugate duals of each other with respect to the $ L^2(X,\mu)$ inner product.
Another possibility consists of exploiting
the lattice of all $L^p(\RN, \ud x)$ spaces, or a subset thereof, in particular a (Gel'fand) triplet of Banach spaces. Some examples have been described above, but clearly more work along these lines is in order.}

\appendix
\section{\hspace*{-4.8mm}ppendix. Lattices of Banach or Hilbert spaces and operators on them}

\def\theequation{\Alph{section}.\arabic{equation}}

 \subsection{Lattices of Banach or Hilbert spaces}
\label{subsec-pip}

For the convenience of the reader, we summarize in this Appendix the basic facts concerning \pip s and operators on them. However, we will restrict the discussion to the simpler case of a lattice  of Banach (LBS) or Hilbert spaces (LHS). Further information may be found in our monograph \cite{pip-book} or our review paper \cite{at-AMP}.

 Let thus $\J = \{ V_p,\, p \in I \}$ be a family of \hs s or reflexive Banach spaces, partially ordered by inclusion. Then $\I$ generates an involutive  lattice $\J$, indexed by $J$, through the operations $(p,q,r \in I)$:
\medskip

\begin{tabular}{lccl}
. involution: &$V_r$                     & $\!\!\!\leftrightarrow\!\!\!$& $V_{\ov{r}}= V_r ^\times $, the conjugate dual of $V_r$\\
. infimum:   & $V_{p \wedge q}$  &$\!\!\! :=\!\!\!$                & $V_p \wedge V_q = V_p \cap V_q$     \\
. supremum: & $V_{p \vee q}$     &$\!\!\! :=\!\!\!$               &$ V_p \vee V_q = V_p + V_q $.
\end{tabular}
\\[2mm]
It turns out that both $V_{p\wedge q}$ and $V_{p\vee q}$ are \hs s, resp. reflexive Banach spaces, under appropriate norms (the so-called projective, resp. inductive norms).
   Assume that the following conditions are satisfied:
  \bei
\item [(i)]  $\I$ contains a  unique self-dual, Hilbert  subspace $V_{o} =V_{\overline{o}}$.

\item [(ii)]  for every $V_r\in\I$, the norm $\|\cdot\|_{\ov{r}}$ on $V_{\ov{r}}=V_{r}^\times$ is the conjugate of the norm $\|\cdot\|_{r}$ on $V_{r}$.
\eni  
In addition to the family $\J =\{V_{r}, \,r\in J\}$, it is convenient to consider the two spaces $V\co$ and $V $ defined as
\be
V = \sum_{q\in I}V_{q}, \quad V\co = \bigcap_{q\in I}V_{q}.
\label{eq:extreme}
\en
These two spaces  themselves usually do \emph{not} belong  to $\I$.

We say that two vectors $f,g\in V$ are \emph{compatible} if  there exists $r \in J  \mbox{ such that } f \in V_{r}, g \in V_{\overline{r}}$\,.
Then a \emph{partial inner product}   on   $V$ is a
   Hermitian form  $\ip{\cdot}{\cdot}$ defined exactly on compatible pairs of vectors.
 In particular, the partial inner product $\ip{\cdot}{\cdot}$ coincides with the inner product of $V_o$ on the latter.
A \emph{partial inner product space}  (\pip)  is a  vector space $V$ equipped with a partial inner product.
Clearly LBSs and LHSs are particular cases of \pip s.

From now on, we will assume that our \pip\  $(V, \ip{\cdot}{\cdot})$ is \emph{nondegenerate},
that is,    $\ip{f}{g} = 0   $ for all $ f \in  V^{\#} $ implies $ g = 0$.  As a consequence,  $(V\co, V)$ and
 every couple $(V_r , V_{\ov r} ), \,  r\in J, $  are a  dual pair in the sense of topological vector spaces \cite{kothe}.
In particular, the original norm topology on $V_r$ coincides with   its Mackey topology $\tau(V_{r},V_{\ov{r}})$, so that indeed its conjugate dual is $(V_r)^\times = V_{\ov {r}}, \; \forall\, r\in J $.
Then,   $r<s$ implies $V_r \subset V_s$, and the embedding operator $E_{sr}: V_r \to V_s$  is continuous and has dense range. In particular, $V\co$ is dense in every $V_{r}$.
 In the sequel, we also assume the partial inner product to be positive definite, $\ip{f}{f}>0$ whenever $f\neq0$.

 A standard, albeit trivial,  example is that of a Rigged Hilbert space (RHS) $\Phi \subset \H \subset \Phi\co$
(it is trivial because the lattice $\I$ contains only three elements).

Familiar concrete examples of \pip s are sequence spaces, with $V = \omega$    the space    of \emph{all} complex sequences $x = (x_n)$, and
spaces of locally integrable functions with $V =L^1_{\rm loc}(\RN, \ud x)$, the space of Lebesgue measurable functions, integrable over compact subsets.

Among LBSs, the simplest example is that of a chain of reflexive Banach spaces. The prototype is the chain 
$ \I = \{L^p := L ^p ( [0,1 ];dx),\; 1 < p < \infty\} $ of Lebesgue spaces over the interval  [0, 1].
\be\label{eq:Lp}
L^{\infty}\;\subset \;\ldots\;\subset   \; L^{\ov{q}}\;\subset\; L^{\ov{r}}\;\subset \; \ldots \; 
 {\subset \;L^{2}\;\subset} \; \ldots \subset \; L^{r}
\subset \; L^{q} \;\subset \;\ldots\;\subset \;L^{1} ,
\en
 where $1<q<r<2$ (of course, $L^{\infty}$ and $L^1$ are not reflexive). Here $L^{q}$ and $ L^{\ov{q}}$ are dual to each other $(1/q + 1/\ov{ q} = 1)$,  and similarly $L^{r}, L^{\ov{r}}\; (1/r + 1/\ov{r} = 1)$.  

As for a LHS, the simplest example is the Hilbert scale generated by a self-adjoint operator $A>I$ in a \hsÊ$\H_o$.
Let $\H_{n}$ be $ D(A^n)$, the domain of $A^n$, equipped with the graph norm $\| f \|_n =  \| A^n f \|, \, f\in D(A^n) $, for $ n \in\NN$ or $n\in \RN^+$, and $\H_{\ov n}:= \H_{-n} =\H_{n}^{\times}$ (conjugate dual):
\be\label{eq:scaleHn}
 \D^{\infty}(A):=\bigcap_{n} \H_n \subset  \ldots \subset   \H_2  \subset   \H_1  \subset  \H_0 \ \subset
    \H_{\ov 1}   \subset \H_{ \ov 2} \ldots \subset   \D_{\ov \infty}(A):=\bigcup_{n} \H_{n}.
\en
 Note that here  the index $n$ may be integer or real, the link between the two cases
being established by the spectral theorem for self-adjoint operators. Here
again the inner product of $\H_0$ extends to each pair $\H_n ,\H_{-n}$, but on $\D_{\ov\infty}(A)$ it yields only a 
\emph{partial} inner product. A standard example is the scale of Sobolev spaces $H^s(\RN), \, s\in \ZN$, in  
$\H_0 = L^2(\RN, dx)$.

\subsection{Operators on  LBSs and LHSs}

Let $V_{J}$   be a    LHS or a LBS. Then  an \emph{operator} on $V_J$  is a map
from a subset $\D (A) \subset V$ into $V$, such that
\smallskip

(i) $\D(A) = \bigcup_{q\in {\sf d}(A)} V_q$, where ${\sf d}(A)$ is a nonempty subset of $J$;
\smallskip

(ii)  For every $q \in  {\sf d}(A )$, there exists $p\in J$ such that the restriction of $A$ to $V_{q}$ is a continuous linear map into $V_{p}$ (we denote this restriction by $A_{pq})$;
\smallskip

(iii) $A$ has no proper extension satisfying (i) and (ii).
\medskip

\noi We denote by Op$(V_J,)$  the set of all operators on   $V_J$ .
 The continuous linear operator $A_{pq}: V_q \to V_{p}$ is called a \emph{representative} of $A$.
The properties of $A$ are conveniently described
 by the set ${\sf j}(A)$ of all pairs $ (q,p )\in  J\times J$ such that $A$ maps $V_{q}$ continuously into $V_{p}$
   Thus the operator $A$ may be identified with   the collection of its representatives,
\be\label{eqj(A)}
A \simeq \{ A_{pq}: V_{q} \to V_{p} : (q,p ) \in  {\sf j}(A)\}.
\en
It is important to notice that  an operator is uniquely determined by \emph{any} of its representatives, in virtue  of Property (iii): there are no extensions for \pip\ operators.

 We will also need the following sets:
\vspace*{-2mm}\begin{align*}
{\sf d}(A) &= \{ q \in J : \mbox{there is a } \,   p \; \mbox{such that}\; A_{pq} \;\mbox{exists} \},
\\
{\sf i}(A) &= \{ p \in J : \mbox{there is a } \, q \; \mbox{such that}\; A_{pq} \;\mbox{exists} \}.
\end{align*}
\\[-4mm]
The following properties are immediate:
\bei
\item [{\bf .}]
${\sf d}(A)$ is an initial subset of $J$:  if $q \in {\sf d}(A)$ and $q' < q$, then $q' \in {\sf d}(A)$, and $A_{pq'} = A_{pq}E_{qq'}$,
 where  $E_{qq'}$ is a representative of the unit operator.

\item [{\bf .}]
${\sf i}(A)$ is a final subset of $J$: if $p \in {\sf i}(A)$ and $p' > p$, then $p' \in {\sf i}(A)$ and $A_{p'q} = E_{p'p} A_{pq}$.
\eni

Although an operator may be identified with a separately continous  sesquilinear form on $V^\# \times V^\#$, or a conjugate linear  continuous map
$V^\#$ into $V$, it is more useful to keep also the \emph{algebraic operations} on operators, namely:
 \bei
\vspace*{-1mm}\item[(i)] \emph{ Adjoint:}
every $A \in\mathrm{Op}(V_J)$ has a {unique} adjoint $A\ta \in \mathrm {Op}V_J)$, defined by
\be\label{eq:adjoint}
\ip {A\ta y} {x} = \ip  {y} { Ax}   , \;\mathrm {for}\,  x \in V_q, \, q\in{\sf d}(A) \;\mathrm {and }\;\, y\in V_{\ov{p}}, \, p \in{\sf i}(A),
\en
that is,    $(A\ta)_{\ov{q}\ov{p}} = (A_{pq})' $, where $(A_{pq})': V_{\ov{p}} \to  V_{\ov{q}}$  is the  adjoint map  of $A_{pq}$.
Furthermore,  one has $A\ta{}\ta = A, $ for every $ A \in {\rm Op}(V_J)$: no extension is allowed, by the maximality condition (iii)  of the definition.

\item[(ii)]
\emph{Partial multiplication:}
 Let $A, B \in   \mathrm{Op}(V_J )$. We say that the product $BA$ is defined if and only  if
there is a $r \in{\sf i}(A) \cap{\sf d}(B)$, that is, if and only if   there is a continuous factorization through some $V_r$:
\be\label{eq:mult}
V_q \; \stackrel{A}{\rightarrow} \; V_r \; \stackrel{B}{\rightarrow} \; V_p , \quad\mbox{{i.e.},} \quad  (BA)_{pq} = B_{pr} A_{rq}, \,\mbox{ for some } \;
q \in{\sf d}(A) , p\in {\sf i}(B).
\en
\eni
{Of particular interest are \emph{symmetric} operators, defined as those operators satisfying the relation
$A\ta  =  A$, since these are the ones that could generate self-adjoint operators in the central \hs, for instance by the celebrated KLMN theorem, suitably generalized to the \pip\ environment \cite[Section 3.3]{pip-book}.}

\section*{Acknowledgement}
This work was partly supported 
by the Istituto Nazionale di Alta Matematica (GNAMPA project ``Propriet\`a spettrali di quasi *-algebre di operatori"). JPA acknowledges gratefully
the hospitality of the Dipartimento di Matematica e Informatica, Universit\`a di Palermo, whereas  CT
acknowledges that of the Institut de Recherche en Math\'ematique et  Physique, Universit\'e catholique de Louvain.


\begin{thebibliography}{99}

 \bibitem{squ-int}S.T. Ali, J-P. Antoine, and J-P. Gazeau, Square integrability of group representations on homogeneous
spaces  I. Reproducing triples and frames, \emph{Ann. Inst. H. Poincar\'e} {\bf 55} (1991) {829--856}

\bibitem{contframes}  S.T. Ali, J-P. Antoine, and J-P. Gazeau, Continuous frames in Hilbert space,
\textit{Annals of Physics} {\bf 222} {(1993)} {1--37}

 \bibitem{ait-book}J-P.~Antoine, A.~Inoue, and C.~Trapani, \emph{Partial *-Algebras and Their Operator Realizations},
Mathematics and Its Applications, vol. 553, Kluwer, Dordrecht, NL, 2002 

\bibitem{pip-book}  {J-P. Antoine and  C. Trapani},  \textit{Partial Inner Product Spaces: Theory and Applications},  Lecture Notes in Mathematics, vol. 1986, Springer-Verlag, Berlin, 2009

\bibitem{at-AMP}{J-P. Antoine  and C. Trapani}   The partial inner product space method: A quick overview,
   \textit{Adv.  in Math. Phys.,}  Vol. {\bf 2010}   (2010)   457635 ;  {Erratum},  \textit{Ibid.} Vol.  {\bf 2011} (2010)   272703.

\bibitem{ant-bal-semiframes1} J-P. Antoine and P. Balazs, Frames and semi-frames, \textit{J. Phys. A: Math. Theor.} {\bf 44}  (2011) {205201};
Corrigendum, \textit{ibid.} {\bf 44} (2011) 479501

\bibitem{ant-bal-semiframes2}   J-P. Antoine and P. Balazs, Frames, semi-frames, and Hilbert scales, \textit{Numer. Funct. Anal. Optimiz.} {\bf 33} (2012) {736--769}

\bibitem{ast-reprodpairs}J-P.~Antoine, M. Speckbacher and  C.~Trapani, Reproducing pairs of measurable functions, preprint, 2016

\bibitem{at-pipops}   J-P.~Antoine and  C.~Trapani,  Operators on partial inner product spaces: Towards a spectral analysis,
\textit{Mediterranean J. Math.}\textit{Mediterranean J. Math.}  {\bf 13}(2016)  323-351


\bibitem{christ} O.~Christensen, \textit{An Introduction to Frames and Riesz Bases\/},  Birkh\"{a}user,  Boston, MA, 2003

 \bibitem{gab-han}  J-P. Gabardo  and D. Han,
Frames associated with measurable spaces,  \textit{Adv. Comput. Math.}  {\bf 18} (2003) 127--147

\bibitem{gould}G.G.~Gould, On a class of integration spaces, \textit{J. London Math. Soc.} \textbf{34} (1959), 161--172

\bibitem{kaiser}  G. Kaiser   \textit{A Friendly Guide to Wavelets},  Birkh{\"a}user, Boston, 1994

\bibitem {kothe} G.~K\"{o}the, \textit{Topological Vector Spaces, Vol. I\/}, Springer-Verlag, Berlin, 1969, 1979

\bibitem{kothe2}  G.~K\"{o}the, \textit{Topological Vector Spaces, Vol. II\/}, Springer-Verlag, Berlin,  1979;  p.84

\bibitem{megg}R.E. Megginson, \textit{An Introduction to Banach Space Theory},  Springer-Verlag, New York-Heidelberg-Berlin, 1998

\bibitem{rahimi}   A. Rahimi,  A. Najati and  Y.N. Dehghan,
Continuous frames in Hilbert spaces,   \textit{Methods Funct. Anal. Topol. }  {\bf 12}   (2006) 170--182

\bibitem {rudin} W. Rudin, \textit{Real and Complex Analysis}, Int. Edition, McGraw Hill, New York et al., 1987; p.73, from Ex.18

 \bibitem{schaefer} H.H. Schaefer, \textit{Topological Vector Spaces},  Springer-Verlag, New York-Heidelberg-Berlin, 1971

 \bibitem {speck-bal}  M. Speckbacher and P. Balazs,  Reproducing pairs and the continuous nonstationary {G}abor transform
 on {LCA} groups, \textit{J. Phys. A: Math. Theor.}, {\bf 48} (2015) 395201

\bibitem{werner} P. Werner, A distributional-theoretical approach to certain Lebesgue and Sobolev spaces,  \textit{J. Math. Anal. Appl.} {\bf 29} (1970)  18--78


\end{thebibliography}
\end{document}